\documentclass[11pt,oneside,english]{amsart}
\usepackage[LGR,T1]{fontenc}
\usepackage[latin9]{inputenc}
\usepackage{geometry}
\usepackage{textcomp}
\usepackage{amsbsy}
\usepackage{amstext}
\usepackage{amsthm}
\usepackage{amssymb}

\makeatletter

\DeclareRobustCommand{\greektext}{%
  \fontencoding{LGR}\selectfont\def\encodingdefault{LGR}}
\DeclareRobustCommand{\textgreek}[1]{\leavevmode{\greektext #1}}
\ProvideTextCommand{\~}{LGR}[1]{\char126#1}

\numberwithin{equation}{section}
\numberwithin{figure}{section}
\theoremstyle{plain}
\newtheorem{thm}{\protect\theoremname}[section]
\theoremstyle{plain}
\newtheorem{lem}[thm]{\protect\lemmaname}
\theoremstyle{remark}
\newtheorem{rem}[thm]{\protect\remarkname}
\theoremstyle{plain}
\newtheorem{prop}[thm]{\protect\propositionname}

\def\makebbb#1{
    \expandafter\gdef\csname#1\endcsname{
        \ensuremath{\Bbb{#1}}}
}\makebbb{R}\makebbb{N}\makebbb{Z}\makebbb{C}\makebbb{H}\makebbb{E}\makebbb{H}\makebbb{P}\makebbb{B}\makebbb{Q}\makebbb{E}

\usepackage{babel}

\usepackage{babel}

\makeatother

\usepackage{babel}

\makeatother

\usepackage{babel}
\providecommand{\lemmaname}{Lemma}
\providecommand{\propositionname}{Proposition}
\providecommand{\remarkname}{Remark}
\providecommand{\theoremname}{Theorem}

\begin{document}
\title{The spherical ensemble and quasi-Monte-Carlo designs}
\author{Robert J. Berman}
\begin{abstract}
The spherical ensemble is a well-known ensemble of $N$ repulsive
points on the two-dimensional sphere, which can realized in various
ways (as a random matrix ensemble, a determinantal point process,
a Coulomb gas, a Quantum Hall state...). Here we show that the spherical
ensemble enjoys remarkable convergence properties from the point of
view of numerical integration. More precisely, it is shown that the
numerical integration rule corresponding to $N$ nodes on the two-dimensional
sphere sampled in the spherical ensemble is, with overwhelming probability,
nearly a quasi-Monte-Carlo design in the sense of Brauchart-Saff-Sloan-Womersley,
for any smoothness parameter $s\leq2.$ The key ingredient is a new
explicit sub-Gaussian concentration of measure inequality for the
spherical ensemble.
\end{abstract}

\maketitle

\section{Introduction}

How to optimally distribute $N$ points on the two-dimensional sphere?
This is a question which has a long history and appears in a wide
range of areas in pure as well as applied mathematics (see the survey
\cite{s-k} and \cite[Section 7]{b-h-s}). The notion of optimality
depends, of course, on the problem at hand. But a recurrent theme
is to distribute the configuration of points $\boldsymbol{x}_{N}:=(x_{1},...,x_{N})\in X^{N}$
so as to minimize 
\[
\left\Vert \delta_{N}(\boldsymbol{x}_{N})-d\sigma\right\Vert 
\]
on $X^{N},$ where $\left\Vert \cdot\right\Vert $ is a given given
a (semi-)norm on the space of all signed measure on the two-sphere
$X,$ $d\sigma$ denotes the standard uniform probability measure
on $X$ and $\delta_{N}(\boldsymbol{x}_{N})$ is the\emph{ empirical
measure} corresponding to $\boldsymbol{x}_{N},$ i.e. the discrete
probability measure on $X$ defined by 
\begin{equation}
\delta_{N}(x_{1},\ldots,x_{N}):=\frac{1}{N}\sum_{i=1}^{N}\delta_{x_{i}}\label{eq:emp measure}
\end{equation}
More precisely, since finding exact minimizers is usually unfeasible,
the aim is typically to distribute the $N$ points $(x_{1},...,x_{N})$
so that $\left\Vert \delta_{N}(\boldsymbol{x}_{N})-d\sigma\right\Vert $
achieves the optimal (minimal) rate as $N\rightarrow\infty$ (as discussed
in the introduction of \cite[I]{l-r-s}). Here we will be concerned
with a notion of optimality which naturally appears~in the context
of numerical integration (cubature) and quasi-Monte-Carlo integration
techniques \cite{c-f,b-s-s-w}, where the norm in question is a Sobolev
norm. 

\subsection{Background }

\subsubsection{(Quasi-)Monte-Carlo integration on cubes}

Monte-Carlo integration is a standard probabilistic technique for
numerically computing the Lebesgue integral of a given, say continuous,
function $f$ over a domain $X$ in Euclidean $\R^{d}$ (or more generally,
a Riemannian manifold $X).$ It consists in generating $N$ random
points $x_{1},...,x_{N}$ in $X,$ with respect to the uniform distribution
$dx$ on $X$ (assuming for simplicity that $X$ has unit-volume)
and approximating 
\[
\int_{X}fdx\approx\frac{1}{N}\sum_{i=1}^{N}f(x_{i})
\]
In other words, the points $x_{1},...,x_{N}$ are viewed as as independent
$\R^{d}-$valued random variables with identical distribution $dx.$
By the central limit theorem the error is of the order $\mathcal{O}(N^{-1/2})$
with high probability:
\begin{equation}
\lim_{N\rightarrow\infty}\P\left(\left|\int_{X}fdx-\frac{1}{N}\sum_{i=1}^{N}f(x_{i})\right|\geq\frac{\lambda}{N^{1/2}}\right)=1-\int_{|y|\geq\lambda}e^{-y^{2}}dx/\pi^{1/2}\label{eq:classical clt}
\end{equation}
if $f$ is normalized to have unit variance. 

The popular\emph{ Quasi-Monte-Carlo method} aims at improving the
order of the convergence, by taking $\boldsymbol{x}_{N}:=(x_{1},...,x_{N})$
to be a judiciously constructed deterministic sequence of $N-$point
configurations on $X.$ In the most commonly studied case when $X$
is the unit-cube $[0,1]^{d}$ in $\R^{d}$ there are well-known explicit
so called \emph{low-descripency sequences }(e.g. digital nets)\emph{
}constructed using the theory of uniform distribution in number theory
\cite{n}, such that 
\begin{equation}
\sup_{f\in C(X):\,V(f)\leq1}\left|\int_{X}fdx-\frac{1}{N}\sum_{i=1}^{N}f(x_{i})\right|\leq C_{d}\frac{(\log N)^{d}}{N}\label{eq:low-discr}
\end{equation}
where $V(f)$ is the\emph{ Hardy-Krause variation} of $f,$ whose
general definition is rather complicated, but for $f$ sufficiently
regular it is, when $d=2,$ given by 
\begin{equation}
V(f):=\int_{[0,1]^{2}}|\frac{\partial^{2}f}{\partial x_{1}\partial x_{2}}|dx+\int_{[0,1]}|\frac{\partial f}{\partial x_{1}}(x_{1},1)|dx_{1}+\int_{[0,1]}|\frac{\partial f}{\partial x_{2}}(1,x_{2})|dx_{2}\label{eq:V f intro}
\end{equation}
This is a consequence of the Koksma-Hlawka inequality, which is the
corner stone of the theory of quasi-Monte-Carlo integration on a cube
\cite{ko,hl,n}. 

\subsubsection{Numerical integration on manifolds}

Let us next recall the general setup for numerical integration on
manifolds, following \cite{b-c-c-g-s-t,b-s-s-w}. Let $X$ be a compact
manifold that we shall take to be two-dimensional. Given a configuration
$\boldsymbol{x}_{N}\in X^{N}$ of $N$ points on $X$ the \emph{worst-case
error} \emph{for the integration rule on $X$ with node set $\boldsymbol{x}_{N}$
with respect to the smoothness parameter $s\in]1,\infty[$ }is defined
by

\begin{equation}
\text{wce }(\boldsymbol{x}_{N};s):=\sup_{f:\,\left\Vert f\right\Vert _{H^{s}(X)}\leq1}\left|\int_{X}fd\sigma-\frac{1}{N}\sum_{i=1}^{N}f(x_{i})\right|\label{eq:def of wce intro}
\end{equation}
where $d\sigma_{g}$ denotes the normalized volume form defined by
$g$ and $\left\Vert f\right\Vert _{H^{s}(X)}$ denotes the norm in
the Sobolev space $H^{s}(X)$ of functions with $s$ fractional derivatives
in $L^{2}(X).$ In other words, 
\[
\text{wce }(\boldsymbol{x}_{N};s)=\left\Vert \delta_{N}(\boldsymbol{x}_{N})-d\sigma_{g}\right\Vert _{H_{0}^{-s}(X)}^{2},
\]
 where $\delta_{N}(\boldsymbol{x}_{N})$ is the empirical measure
\ref{eq:emp measure} and $H_{0}^{-s}(X)$ denotes the Sobolev space
of all mean zero distributions on $X,$ endowed with the Hilbert norm
which is dual to the smoothness parameter $s\in]1,\infty[$ (see Section
\ref{subsec:Sobolev-spaces-and}). The role of the Hardy-Krause variation
norm \ref{eq:V f intro} on a Euclidean square will in the present
two-dimensional Riemannian setting be played by the Sobolev norm with
smoothness parameter $2:$ 
\[
\left\Vert f\right\Vert _{H^{2}(X)}:=\left(\int_{X}|\Delta_{g}f|^{2}dV_{g}\right)^{1/2},
\]
 where $\Delta_{g}$ denotes Laplace operator on $C^{\infty}(X).$
The worst case error $\text{wce }(\boldsymbol{x}_{N};s)$ is also
called the \emph{generalized discrepancy }\cite{c-f}\emph{ }because
of the similarity with the Koksma-Hlawka inequality on a cube. A sequence
$\boldsymbol{x}_{N}\in X^{N}$ is said to be of \emph{convergence
order $\mathcal{O}(N^{-\kappa})$ with respect to the smoothness parameter
$s$} if 
\[
\text{wce }(\boldsymbol{x}_{N};s)\leq\mathcal{O}(N^{-\kappa})
\]
The optimal convergence order is $\mathcal{O}(N^{-s/2}).$ More precisely,
by \cite[Thm 2.14]{b-c-c-g-s-t}, there exists a positive constant
$c(s)$ such that for any sequence $\boldsymbol{x}_{N}\in X^{N}$
\begin{equation}
\text{wce }(\boldsymbol{x}_{N};s)\geq c(s)N^{-s/2}.\label{eq:wce for speci}
\end{equation}

\subsubsection{Quasi-Monte Carlo designs on the two-sphere}

Consider now the case when $X$ is the two-dimensional sphere endowed
with the Riemannian metric induced from the standard embedding of
$X$ as the unit-sphere in Euclidean $\R^{3}.$ We will denote by
$d\sigma$ the probability measure on $X$ obtained by normalizing
the area form of $g.$ Following \cite{b-s-s-w} a sequence of $N-$point
configurations $\boldsymbol{x}_{N}\in X^{N}$ is said to be a sequence
of \emph{Quasi-Monte-Carlo designs (QMC) wrt the smoothness parameter}
$s\in]1,\infty[,$ if the corresponding worst case errors $\text{wce }(\boldsymbol{x}_{N};s)$
have optimal convergence order, i.e. if 
\[
\text{wce }(\boldsymbol{x}_{N};s)=\mathcal{O}(N^{-s/2})
\]
In particular, this convergence is faster than the one offered by
the standard Monte-Carlo method. Indeed, as recalled above, Monte-Carlo
integration gives, with high probability, an error of the order $N^{-1/2}$
for a fixed function $f,$ even if the function is smooth. 

The notion of a QMC design is modeled on the influential notion of
a \emph{spherical t-design} $\boldsymbol{x}_{N}\in X^{N},$ introduced
in \cite{d-g-s}. In fact, as shown in \cite[Thm 6]{b-s-s-w}, it
follows from the solution of the Korevaar-Meyers conjecture in \cite{b-r-v},
that there exists a sequence of spherical $t-$designs $X^{N}$ with
$t$ of the order $N^{1/2},$ which is a QMC design for any $s\in]1,\infty[.$
Moreover, for a fixed $s\in]1,2[$ reproducing kernel techniques reveal
that any sequence of maximizers $\boldsymbol{x}_{N}(s)\in X^{N}$
of the generalized sum
\[
\sum_{i,j\leq N}\left|x_{i}-x_{j}\right|^{2s-2}
\]
is a QMC design wrt the smoothness parameter $s$ (see \cite{b-s-s-w}). 

However, all the sequences of QMC designs discussed above are non-explicit
for $N$ large. Moreover, it is very challenging to approximate them
numerically for large $N,$ due, in particular, to an abundance of
local minima of the corresponding functionals to be minimized on $X^{N}$
\cite{g-p,w}. One is thus lead to wonder if probabilistic methods
can be used to improve the convergence order of the standard Monte-Carlo
method by taking the points $x_{1},...,x_{N}$ on the sphere to be
appropriately correlated, as in repulsive particle systems? A natural
class of such point processes is offered by the class of \emph{determinantal
point processes}, whose utility for Monte-Carlo type numerical integration
was advocated in \cite{b-h}. The main aim of the present work is
to show that the a particular determinantal point process on the two-sphere
known as \emph{the} \emph{spherical ensemble} enjoys quite remarkable
converge properties from the point of view of numerical integration,
for any smoothness parameter $s\in]1,2].$

\subsection{\label{subsec:Main-results-for}Main results for the spherical ensemble}

The spherical ensemble first appeared as a \emph{Coulomb gas, }also
known as a one-component plasma, in the physics literature (see the
monograph \cite{fo} and references therein). We recall that the Coulomb
gas on the two-sphere $X$ with $N$ particles, at inverse temperature
$\beta,$ is defined by the following symmetric probability measure
on $X^{N}:$ 
\begin{equation}
d\P_{N,\beta}:=\frac{1}{Z_{N,\beta}}e^{-\beta E^{(N)}}d\sigma^{\otimes N},\,\,\,\,E^{(N)}:=-\sum_{i\neq j\leq N}\frac{1}{2}\log\left|x_{i}-x_{j}\right|\label{eq:def of dP and E N intro}
\end{equation}
where $X$ has been embedded as the unit-sphere in Euclidean $\R^{3}.$
It represents the microscopic state of $N$ unit charge particles
in thermal equilibrium on $X,$ interacting by the Coulomb energy
$E^{(N)},$ subject to a neutralizing uniform back-ground charge.
More precisely, the spherical ensemble $(X,\P_{N})$ coincides with
Coulomb gas on the sphere at the particular inverse temperature $\beta=2,$
for which the Coulomb gas becomes a determinantal point process \cite{fo,h-k-p}.
An elegant random matrix realization of the spherical ensemble was
exhibited in \cite{kr}. Consider two rank $N$ complex matrices $A$
and $B$ and take their entries to be standard normal variables. Then
the spherical ensemble coincides with the random point process defined
by the eigenvalues of $AB^{-1}$ in the complex plane, scaled by $1/N^{1/2}$
and mapped to the two-sphere, using stereographic projection. 

In the present work it is shown that a random $N-$point configuration
$\boldsymbol{x}_{N}$ in the spherical ensemble is, with overwhelming
probability, nearly a Quasi-Monte Carlo design for any $s\in]1,2]:$ 
\begin{thm}
\label{thm:main intro}Consider the spherical ensemble with $N$ particles.
There exists a constant $C$ such that for any given $R$ in $[(\log N)^{-1/2},N(\log N)^{-1/2}]$
\[
\P_{N}\left(\text{wce }(\boldsymbol{x}_{N};2)\leq R\frac{(\log N)^{1/2}}{N}\right)\geq1-\frac{1}{N^{R^{2}/C-C}}
\]
As a consequence, for any $s\in]1,2],$ there exists a constant $C(s)$
such that for any given $R$ in $[(\log N)^{-1/2},N(\log N)^{-1/2}]$
\[
\P_{N}\left(\text{wce}(\boldsymbol{x}_{N};s)\leq R^{s/2}\frac{(\log N)^{s/4}}{N^{s/2}}\right)\geq1-\frac{1}{N^{R^{2}/C(s)-C}}.
\]
\end{thm}

That the estimate in the case $s=2$ implies the one for $s<2$ follows
from \cite[Lemma 26]{b-s-s-w}. Note that the worst case error in
the case $s=2$ is similar to the worst case error of low-discrepancy
sequences on the cube (formula \ref{eq:low-discr}). It should, however,
be stressed that on the two-sphere there are no explicitly constructed
sequences $(\boldsymbol{x}_{N})$ saturating the optimal rate $\text{wce }(\boldsymbol{x}_{N};2)=\mathcal{O}(N^{-1}),$
even if logarithmic factors are included (see the discussion in Section
\ref{subsec:Explixit-sequences-for}). 

An important feature of the proof of the previous theorem is that
it yields attractive values on the constants in question. Indeed,
for any $\eta>2/\log N$ the following explicit bound is obtained:
\begin{equation}
\P_{N}\left(\text{wce }(\boldsymbol{x}_{N};2)\leq\frac{e}{(8\pi)^{1/2}}\frac{\left((1+\eta)\log N\right)^{1/2}}{N}\right)\geq1-\frac{(2e\eta\log N)^{2}}{N^{\eta}},\label{eq:explicit intro}
\end{equation}
where $\log N$ denotes the natural logarithm, assuming that $N\geq1000$
(otherwise, Euler's number $e$ appearing in the left hand side has
to be replaced by a slightly larger constant). For example, when $N=1000$
this yields the worst-case-error bound $\text{wce }(\boldsymbol{x}_{N};2)<0.0025$
with more than $99\%$ confidence (by taking $\eta=2)$. Moreover,
$N=10000$ yields $\text{wce }(\boldsymbol{x}_{N};2)<0.0003$ with
$99.99\%$ confidence.

The previous theorem should be compared with the conjecture in \cite{b-s-s-w},
supported by numerical simulations, saying that minimizers $\boldsymbol{x}_{N}$
of the logarithmic energy $E^{(N)}$ (formula \ref{eq:def of dP and E N intro})
are QMC designs for $s\in]1,3].$ However, a practical advantage of
the spherical ensemble is that it can be simply generated by employing
$\mathcal{O}(N^{3})$ elementary operations (using its random matrix
representations), while no polynomial time algorithm for constructing
near-minimizers of $E^{(N)}$ is known \cite[Problem 7]{sm}. Concerning
the sharpness of the inequalities in the previous theorem we note
that the restriction to $s\leq2$ is necessary, as follows from formula
\ref{eq:clt of r-v} below. Moreover, the power $1/2$ of $\log N$
appearing in the inequalities for $s=2$ can be expected to be optimal. 

Theorem \ref{thm:main intro} will be deduced from a new concentration
of measure inequality in Sobolev spaces (Theorem \ref{thm:concentr of measure for spherical}
below), which, in turn, will follow from the following bound on the
moment generating function of the square of the random variable $\text{wce }(\boldsymbol{x}_{N};s).$
\begin{thm}
\label{thm:moment bound for wce}For any $\epsilon>0$ and $\alpha\in]0,4\pi[$
\[
\E_{N}\left(e^{\alpha N^{2}\left(\text{wce }(\boldsymbol{x}_{N};2+\epsilon)\right)^{2}}\right)\leq\left(\det(I-\frac{\alpha}{4\pi}\Delta_{g}^{-(1+\epsilon)})\right)^{-1/2}<\infty
\]
 where $\Delta_{g}$ denotes the Laplace operator on $X$ and $\det(I-\lambda\Delta_{g}^{-(1+\epsilon)})$
is the Fredholm (spectral) determinant of its fractional power $\Delta_{g}^{-(1+\epsilon)}$
(see Section \ref{subsec:Sobolev-spaces-and}). 
\end{thm}

In fact, this an asymptotic equality, as $N\rightarrow\infty$ (as
will be shown elsewhere \cite{berm17}). Combining Theorem \ref{thm:moment bound for wce}
with some spectral theory the following concentration of measure inequality
for the spherical ensemble is obtained:
\begin{thm}
\label{thm:concentr of measure for spherical}There exist explicit
constants $A_{1},A_{2}$ and $A_{3}$ such that for any positive integer
$N$ and $\epsilon>0$ 
\[
\P_{N}\left(\left\Vert \delta_{N}-d\sigma\right\Vert _{H_{0}^{-(2+\epsilon)}(X)}>\delta\right)\leq e^{-A_{1}N^{2}\delta^{2}+\frac{A_{2}}{\epsilon}+A_{3}}.
\]
\end{thm}

\subsection{Outlook on the the case of general compact surfaces}

The results for the two-sphere can be generalized to any compact two-dimensional
Riemannian manifold $X.$ Here we will just highlight the main points,
deferring details to \cite{berm17}. Given a Riemannian surface $(X,g)$
of strictly positive genus denote by $g_{c}$ the unique Riemannian
metric on $X$ with constant curvature which is conformally equivalent
to $g.$ To $(X,g_{c})$ one can attach a canonical $N-$particle
determinantal point process $(X^{N},d\P_{N}),$ which can be viewed
as a higher genus generalization of the spherical ensemble \cite{berm3}.
In this general setting the bound in Theorem \ref{thm:moment bound for wce}
holds up to multiplying the right hand side with a factor of the form
$(1+e^{-\delta/N}),$ for an explicit positive constant $\delta,$
depending on the injectivity radius of $(X,g_{c}).$ This is shown
in essentially the same way as in the spherical setting, using the
general Moser-Trudinger type inequalities in \cite{berm3} as a replacement
for the inequalities \ref{eq:spherical m-t intro} recalled below.
The analogs of Theorems \ref{thm:main intro}, \ref{thm:concentr of measure for spherical}
then follow as before. In particular, 
\[
\left\Vert \frac{1}{N}\sum_{i=1}^{N}\delta_{x_{i}}-dV_{g_{c}}\right\Vert _{H_{0}^{-2}(X)}=\mathcal{O}\left(\frac{(\log N)^{1/2}}{N}\right)
\]
holds with probability $1-\mathcal{O}(1/N^{\infty})$ (in the sense
of Theorem \ref{thm:main intro}). Expressed in terms of the original
Riemannian metric $g$ this means that introducing the ``weight function''
\[
w:=dV_{g}/dV_{g_{c}}
\]
and sampling a configuration $(x_{1},...,x_{N})$ in the canonical
$N-$particle ensemble $(X^{N},d\P_{N})$
\begin{equation}
\P_{N}\left(\sup_{f:\,\left\Vert f\right\Vert _{H_{0}^{2}(X)}\leq1}\left|\int_{X}fdV_{g}-\frac{1}{N}\sum_{i=1}^{N}f(x_{i})w(x_{i})\right|\leq\mathcal{O}\left(\frac{(\log N)^{1/2}}{N}\right)\right)\geq1-\mathcal{O}(1/N^{\infty})\label{eq:prob for general g}
\end{equation}
(in the sense of Theorem \ref{thm:main intro}). In fact, the original
Riemannian metric $g$ also induces a determinantal $N-$particle
point process on $X.$ In physical terms, the corresponding probability
measure $d\P_{g}^{N}$ on $X^{N}$ represents the probability density
for an integer Quantum Hall state, i.e. an $N-$particle state of
electrons confined to $(X,g)$ subject to a constant magnetic field,
whose strength is proportional to $N.$ However, as explained in \cite{berm17},
the error in the corresponding estimate \ref{eq:prob for general g}
will be of the larger order $\mathcal{O}(1/N^{1/2})$ (unless $g$
has constant curvature). Accordingly, replacing the original metric
$g$ with the constant curvature one $g_{c}$ is analogous to the
use of \emph{importance sampling} in the standard Monte-Carlo method.
Recall that the latter method amounts to calculating integrals $\int_{X}fdV_{g}$
by taking the points $x_{i}$ to be independent realizations of a
``target measure'' $\nu$ (taken to be different than $dV_{g}$
with the aim of reducing the variance; compare the discussion in \cite[Section 1.2]{b-h})

It should be stressed, however, that one advantage of the spherical
setting is that all the constants can be explicitly estimated, while,
in the general setting, the constants depend on spectral invariants
of $(X,g_{c}).$ Moreover, from a practical point of view the random
matrix realization of the spherical ensemble offers a convenient implementation
algorithm, while the general algorithm for simulating determinantal
point processes \cite{h-k-p} has to be employed for a general surface
$X$ (which, loosely speaking, replaces the task of finding the $N$
eigenvalues with Gram-Schmidt orthogonalization).

\subsection{Outline of the proof of Theorem \ref{thm:main intro}}

As shown in \cite{r-v0}, for a fixed function $f$ on $X$ the following
Central Limit Theorem (CLT) holds for the spherical ensemble: for
any $f\in H^{1}(X),$ normalized so that $\int|\nabla_{g}f|^{2}dV_{g}=4\pi,$
\begin{equation}
\lim_{N\rightarrow\infty}\P_{N}\left(\left|\frac{1}{N}\sum_{i=1}^{N}f(x_{i})-\int_{X}fd\sigma\right|\geq\frac{\lambda}{N}\right)=1-\int_{|y|\geq\lambda}e^{-y^{2}}dx/\pi^{1/2}\label{eq:clt of r-v}
\end{equation}
A key ingredient in the proof of Theorem \ref{thm:main intro}, given
in Section \ref{sec:Concentration-of-measure}, is the following quantitative
refinement of the previous CLT, obtained in \cite{berm3}: 
\[
\P_{N}\left(\frac{1}{N}\sum_{i=1}^{N}f(x_{i})-\int_{X}fd\sigma\leq\frac{\lambda}{N}\right)\leq e^{-\lambda^{2}/2}
\]
More precisely, the following slightly stronger dual bound on the
moment generating function was established in \cite{berm3} (using
complex differential geometry): 

\begin{equation}
\E_{N}\left(\exp N(N+1)\left(\frac{1}{N}\sum_{i=1}^{N}f(x_{i})-\int_{X}fd\sigma\right)\right)\leq e^{N(N+1)\lambda^{2}/2}\label{eq:spherical m-t intro}
\end{equation}
(coinciding when $N=1$ with the well-known sharp Moser-Trudinger
inequality on the two-sphere). Note, however, that these inequalities
only hold for a\emph{ fixed} normalized function $f\in H^{1}(X)$
and fail drastically for the random variable $\text{wce }(\boldsymbol{x}_{N};1)$
obtained by taking the sup over all normalized $f\in H^{1}(X).$ Indeed,
$\text{wce }(\boldsymbol{x}_{N};1)=\infty$ on all of $X^{N}$ since
$H^{1}(X)$ contains unbounded functions (recall that $s=1$ is the
borderline case for the Sobolev embedding of $H^{s}(X)$ into $C(X)).$ 

Here we will interpret the inequality \ref{eq:spherical m-t intro}
as the statement that the random variable 
\[
Y_{N}:=\sum_{i=1}^{N}\delta_{x_{i}}-Nd\sigma
\]
 taking values in the dual Sobolev space $H^{-(2+\epsilon)}(X)$ is
\emph{sub-Gaussian} wrt a canonical Gaussian random variable $G$
on $H^{-(2+\epsilon)}(X)$ (see Remark \ref{rem:gff}). Using some
basic Gaussian measure theory and spectral theory for the Laplacian
we then deduce the moment bound in Theorem \ref{thm:moment bound for wce},
which, in turn, implies the concentration of measure inequality Theorem
\ref{thm:concentr of measure for spherical}. Finally, we show that
the latter inequality, implies Theorem \ref{thm:main intro}, when
combined with the results in \cite{b-c-c-g-s-t,b-s-s-w} relating
$\text{wce }(\boldsymbol{x}_{N};s)$ corresponding to different values
of $s.$ 

\subsection{Further comparison with previous results}

As shown in \cite[Thm 1]{Hi}, building on \cite{a-z}, the spherical
ensemble satisfies the following asymptotics: for any $s\in]1,2[$
there exists a positive constant $C(s)$ such that
\begin{equation}
\sqrt{\E\left(\text{wce }(\boldsymbol{x}_{N};s)^{2}\right)}=C(s)N^{s/2}+o(N^{s/s})\label{eq:asympt of E av wce for spherical}
\end{equation}
This result should be compared with \cite[Thm 24]{b-s-s-w}, which
says that if $X$ is partitioned into $N$ equal area regions whose
diameters are bounded by $CN^{-1/2}$ and a sequence $\boldsymbol{x}_{N}$
of $N$ point is randomly chosen from $N$ different regions, then
the corresponding $\sqrt{\E\left(\text{wce }(\boldsymbol{x}_{N};s)^{2}\right)}$
is also of the order $\mathcal{O}(N^{s/2}),$ when $s\in]1,2[.$ However,
in contrast to Theorem \ref{thm:main intro}, the results in \cite{Hi,b-s-s-w},
referred to above, do not give any information about the probability
that the worst-case-error $\text{wce }(\boldsymbol{x}_{N};s)$ for
a random sequence $\boldsymbol{x}_{N}$ in the corresponding ensembles
is close to the average worst-case-error. The only previous result
in this direction appears to be \cite[Thm 1.1]{a-z}, saying that
for any $M>0$ there exists $C_{M}>0$ such that 
\begin{equation}
\P\left(D_{L^{\infty}}^{C}(\boldsymbol{x}_{N})\leq C_{M}\frac{(\log N)^{1/2}}{N^{3/4}}\right)\geq1-\frac{1}{N^{M}}.\label{eq:ineq for L infty sph cap}
\end{equation}
where $D_{L^{\infty}}^{C}(\boldsymbol{x}_{N})$ is the\emph{ $L^{\infty}-$spherical
cap discrepancy} defined by 
\[
D_{L^{\infty}}^{C}(\boldsymbol{x}_{N}):=\sup_{f=1_{\mathcal{C}}}\left|\int_{X}fd\sigma-\frac{1}{N}\sum_{i=1}^{N}f(x_{i})\right|
\]
where the sup if taken over all characteristic $f$ functions of the
form $f=1_{\mathcal{C}},$ where $\mathcal{C}$ is a spherical cap
in the two-sphere $X,$ i.e. the intersection of $X$ with a half-space
in $\R^{3}$ (the proof of the inequality  \ref{eq:ineq for L infty sph cap}
is based on a variance estimate in \cite{a-z}). Since 
\begin{equation}
\text{wce }(\boldsymbol{x}_{N};\frac{3}{2})\leq aD_{L^{\infty}}^{C}(\boldsymbol{x}_{N})\label{eq:wce smaller than D infty}
\end{equation}
for an explicit constant $a$ \cite[Page 16]{b-s-s-w}), the inequality
\ref{eq:ineq for L infty sph cap} implies that

\begin{equation}
\P\left(\text{wce }(\boldsymbol{x}_{N};\frac{3}{2})\leq C_{M}\frac{(\log N)^{1/2}}{N^{3/4}}\right)\geq1-\frac{1}{N^{M}}.\label{eq:ineq for wce in az}
\end{equation}
This is a bit weaker then the case $s=3/2$ of Theorem \ref{thm:main intro}
(where the power of $\log N$ is $3/8(<1/2)$ and moreover the dependence
of $C_{M}$ on $M$ is made explicit). We recall that the inequality
\ref{eq:wce smaller than D infty} follows from the fact that $\text{wce }(\boldsymbol{x}_{N};\frac{3}{2})$
is comparable to the\emph{ }$L^{2}-$spherical cap discrepancy 
\[
D_{L^{2}}^{C}(\boldsymbol{x}_{N}):=\int_{\mathcal{}}\left|\int_{X}fd\sigma-\frac{1}{N}\sum_{i=1}^{N}f(x_{i})\right|^{2}Df,\,\,\,f=1_{\mathcal{C}}
\]
where $Df$ is a certain probability measure measure on the space
of all spherical caps $\mathcal{C}$ \cite[Page 16]{b-s-s-w}.

\subsubsection{\label{subsec:Explixit-sequences-for}Explicit sequences for numerical
integration on the sphere }

As recalled above, $\text{wce }(\boldsymbol{x}_{N};\frac{3}{2})$
is comparable to the\emph{ }$L^{2}-$spherical cap discrepancy $D_{L^{2}}^{C}(\boldsymbol{x}_{N}).$
In \cite{l-r-s} the representation theory of Hecke operators and
modular forms was used to obtain an explicit sequence satisfying the
bound $D_{L^{2}}^{C}(\boldsymbol{x}_{N})\leq CN^{-1/2}\log N$ (see
\cite[I, Thm 2.2]{l-r-s}). The proof of the bound uses Deligne's
proof of the Weil conjectures and also yields, as explained in \cite[Remark 3.10]{b-c-c-g-s-t},
$\text{wce }(\boldsymbol{x}_{N};s)\leq CN^{-1/2}\log N$ for any $s>1.$
However, these rates are only close to optimal as $s$ approaches
$1.$ A different sequence satisfying $D_{L^{2}}^{C}(\boldsymbol{x}_{N})\leq CN^{-1/2}(\log N)^{1/2}$
was then constructed in \cite{b-d}, by mapping a digital net on the
square to $X.$ Numerical evidence was provided in \cite{b-d} indicating
that the latter sequence has the optimal rate $\mathcal{O}(N^{-3/4}).$
See also \cite[Section 8]{b-s-s-w} for numerical experiments for
a range of different classes of point sets on the two-sphere. 

\subsubsection{Concentration of measure}

It may be illuminating to compare Theorem \ref{thm:main intro}for
$s=2$ with the concentration of measure inequalities for\emph{ independent}
random variables established in \cite{b-g-v}, which can be viewed
as a quantitative refinement of the classical CLT \ref{eq:classical clt}.
In the particular case of standard Monte-Carlo integration on a cube
the inequalities in \cite{b-g-v} imply that there exists a constant
$C$ such that for any $R>C$
\begin{equation}
\P_{N}\left(\sup_{\left\Vert \nabla f\right\Vert _{L^{\infty}}\leq1}\left|\int_{X}fdx-\frac{1}{N}\sum_{i=1}^{N}f(x_{i})\right|\leq R\frac{(\log N)^{1/2}}{N^{1/2}}\right)\geq1-\frac{1}{N^{R^{2}/C}}\label{eq:ineq in villani et al}
\end{equation}
(see also \cite{bo} for a simplified proof). This inequality thus
exhibits the smaller denominator $1/N^{1/2},$ due the points $x_{i}$
being independent random variables. Moreover, the role of the Sobolev
norm $W^{1,2}$ appearing for $s=2$ is in the inequality \ref{eq:ineq in villani et al}
played by the $W^{1,\infty}-$norm $\left\Vert \nabla f\right\Vert _{L^{\infty}}$.
The proof uses the dual representation of the $W^{1,\infty}-$norm
between probability measures as the $L^{1}-$Wasserstein metric (aka
Monge-Kantorovich distance) which fits into the general setting of
optimal transport theory. We also recall that in the particular case
when $d=1$ the sharp form of the Dvorestky-Kiefer-Wolfowitz inequality
for $N$ independent real random variables (motivated by the Kolmogorov-Smirnov
test for goodness of fit in statistics) \cite{ma} yields
\[
\P_{N}\left(\sup_{\left\Vert \nabla f\right\Vert _{L^{\infty}}\leq1}\left|\int_{X}f\mu-\frac{1}{N}\sum_{i=1}^{N}f(x_{i})\right|\geq\lambda\right)\leq2e^{-2N\lambda^{2}}
\]
for any probability measure $\mu$ on $\R$ with a continuous density
(see the discussion in \cite[page 2304-2305]{bo}). Generalizations
of the concentration of measure inequality \ref{eq:ineq in villani et al}
to general Coulomb (and Riesz) gas ensembles $(d\P_{N,\beta},\R^{N})$
in Euclidean $\R^{N}$ have been obtained in \cite{R-S,c-h-m} and
on compact Riemannian manifolds in \cite{ga-z}. In particular, in
the case of the spherical ensemble the inequalities in \cite{ga-z}
say that

\begin{equation}
\P_{N}\left(\sup_{\left\Vert \nabla f\right\Vert _{L^{\infty}}\leq1}\left|\int_{X}fdx-\frac{1}{N}\sum_{i=1}^{N}f(x_{i})\right|\leq\delta\right)\geq1-e^{-\frac{1}{4\pi}\frac{\delta^{2}}{2}+\frac{1}{4\pi}\frac{\log N}{N}+C\frac{1}{N}}\label{eq:Garcia ine}
\end{equation}
To see the relation to the present $L^{2}-$setting note that the
Sobolev inequality shows that, in dimension $d=2,$ 
\[
\left\Vert \nabla f\right\Vert _{L^{\infty}}(X)\leq C_{\epsilon}\left\Vert \nabla f\right\Vert _{H^{2+\epsilon}(X)}
\]
for any $\epsilon>0$ (where the constant $C_{\epsilon}$ blows up
as $\epsilon\rightarrow0).$ Hence, the inequality \ref{eq:Garcia ine}
implies a concentration inequality for Sobolev $H^{2+\epsilon}(X)-$norms
which is similar to the inequality in Theorem \ref{thm:concentr of measure for spherical}).
However, the main virtue of Theorem \ref{thm:concentr of measure for spherical}
is that, for a fixed $\epsilon>0,$ there is no $N-$dependent sub-dominant
error terms in the right hand side. This allows one to apply Theorem
\ref{thm:concentr of measure for spherical} to $\delta$ of the order
$N^{-1}$ (modulo logarithmic factors), while one can at best take
$\delta$ of the order $N^{-1/2}$ in the inequality \ref{eq:Garcia ine}.

\subsection{Acknowledgments}

Thanks to Klas Modin for comments on a draft of the paper. This work
was supported by grants from the KAW foundation, the Göran Gustafsson
foundation and the Swedish Research Council.

\section{\label{sec:Concentration-of-measure}Spectral preparations}

We will denote by $\P_{N}$ and $\E_{N}$ the probabilities and expectations,
respectively, defined wrt the spherical ensemble with $N-$particles
$(X^{N},d\P_{N})$ (whose definition was recalled in Section \ref{subsec:Main-results-for}).
We start with some preliminaries.

\subsection{\label{subsec:Sobolev-spaces-and}Sobolev spaces and spectral theory}

Let us first consider a general setup of Sobolev spaces on a compact
Riemannian manifold $(X,g).$ Denote by $\left\langle \cdot,\cdot\right\rangle _{L^{2}}$
the corresponding scalar product on $C^{\infty}(X):$ 
\[
\left\langle u,v\right\rangle _{L^{2}}:=\int_{X}uvdV_{g},
\]
 where $dV_{g}$ denotes the Riemannian volume form (we will denote
by $d\sigma_{g}$ the probability measure obtained by normalizing
$dV_{g}).$ We will denote by $\Delta_{g}$ the Laplace operator on
$C^{\infty}(X),$ with the sign convention which makes $\Delta_{g}$
a densely defined \emph{positive }symmetric operator on $L^{2}(X,dV_{g}):$
\[
\left\langle \Delta_{g}u,v\right\rangle _{L^{2}}:=\int_{X}g(\nabla_{g}u,\nabla_{g}v)dV_{g},
\]
where $\nabla_{g}u$ denotes the gradient of $u$ wrt $g.$ By the
spectral theorem, for any $p\in\R$ the $p$th power $\Delta_{g}^{p}$
is a densely defined operator on $L^{2}(X,dV_{g}).$ 

Fix a ``smoothness parameter'' $s,$ assumed to be strictly positive: 
\begin{itemize}
\item $H^{s}(X)/\R$ is defined as the completion of $C^{\infty}(X)/\R$
with respect to the scalar product defined by 
\begin{equation}
\left\langle u,u\right\rangle _{s}:=\int_{X}\Delta_{g}^{s/2}u\Delta_{g}^{s/2}udV_{g}=\int_{X}u\Delta^{s}udV_{g}\label{eq:def of scalar s}
\end{equation}
\item $H_{0}^{-s}(X)$ is defined as the sub-space of all distributions
$\nu$ on $X$ such that $\left\langle \nu,1\right\rangle =0$ satisfying
\[
\left\langle \nu,\nu\right\rangle _{-s}:=\sup_{u\in C^{\infty}(X)}\frac{\left\langle \nu,u\right\rangle }{\left\langle u,u\right\rangle _{s}}<\infty
\]
\end{itemize}
Here we view a distribution $\nu$ on $X$ as an element in the linear
dual of the vector space $C^{\infty}(X).$ We endow $H^{s}(X)/\R$
and $H_{0}^{-s}(X)$ with the Hilbert space structures defined by
the scalar products $\left\langle \cdot,\cdot\right\rangle _{s}$
and $\left\langle \cdot,\cdot\right\rangle _{-s},$ respectively.
Note that the norm on $H^{s}(X)/\R$ is increasing wrt $s,$ while
the norm on $H_{0}^{-s}(X)$ is decreasing wrt $s.$ Moreover, by
definition, we have that 
\begin{equation}
\left\Vert \delta_{N}(\boldsymbol{x}_{N})-d\sigma_{g}\right\Vert _{H_{0}^{-s}(X)}^{2}=\text{wce }(\boldsymbol{x}_{N};s)\label{eq:sobolev is wce}
\end{equation}
where $\text{wce }(\boldsymbol{x}_{N};s)$ is the worst-case error
for the integration rule on $X$ with node set $\boldsymbol{x}_{N}$
with respect to the smoothness parameter $s\in]1,\infty[$ (defined
by formula \ref{eq:def of wce intro}). By the Sobolev embedding theorem
$\text{wce }(\boldsymbol{x}_{N};s)$ is finite precisely when $s>\dim X/2,$

By duality the operator $\Delta_{g}$ is also defined on the space
of all distributions $\nu:$
\[
\left\langle \Delta_{g}\nu,u\right\rangle :=\left\langle \nu,\Delta_{g},u\right\rangle 
\]
 The following lemma follows directly from the definition of the Hilbert
spaces in question:
\begin{lem}
\label{lem:isometry}The operator $\Delta_{g}$ induces an isometry
when restricted to $C^{\infty}(X)/\R$
\[
H^{s}(X)/\R\rightarrow H^{s-2}(X)/\R,\,\,H_{0}^{-s}(X)\rightarrow H_{0}^{-s-2}(X)
\]
\end{lem}

Next, recall that, by the spectral theorem, the set of eigenfunctions
$f_{i}$ of $\Delta_{g}$ in $C^{\infty}(X)$ form and orthonormal
bases for $L^{2}(X,dV_{g}).$ The following lemma then follows directly
by duality:
\begin{lem}
\label{lem:expansion in terms of eig}There exists an orthonormal
basis $\nu_{i}$ in the Hilbert space $\left\langle H_{0}^{-s}(X),\left\langle \cdot,\cdot\right\rangle _{-s}\right\rangle $
such that 
\[
\nu_{i}=f_{i}dV_{g}
\]
 (acting on $C^{\infty}(X)$ by integration) where $f_{i}$ runs over
all eigenfunctions of the Laplacian on $C^{\infty}(X)$ with strictly
positive eigenvalues $\lambda_{i}.$ As a consequence, if $f_{i}$
is normalized so that $\left\Vert f_{i}\right\Vert _{L^{2}}=1$ and
\[
\nu=\sum_{i=1}^{\infty}c_{i}f_{i}
\]
 in $H_{0}^{-s}(X)$ then 
\[
\left\langle \nu,\nu\right\rangle _{-s}:=\sum_{i=1}^{\infty}\lambda_{i}^{-s}c_{i}^{2}
\]
\end{lem}

\begin{rem}
\label{rem:equiv norms}In the literature different Sobolev space
norms on $H^{s}(X)/\R$ are often used, for example, obtained by replacing
$\Delta^{s}$ in the last equality in formula \ref{eq:def of scalar s}
with $(I+\Delta)^{s}$ (as in \cite{b-c-c-g-s-t,b-s-s-w}) or, more
generally, any other elliptic pseudodifferential operator $P_{s}$
of order $s.$ \cite{c-f}  Anyway, the norms on $H^{s}(X)/\R$ defined
by any two such operators are quasi-isometric, by elliptic regularity
theory (see the discussions in \cite{c-f,b-s-s-w}). Hence, when the
norm is changed Theorem \ref{thm:concentr of measure for spherical}
still applies if $\delta$ is replaced by $C(\epsilon)\delta$ for
a positive constant $C(\epsilon)$ (and similarly for Theorems \ref{thm:moment bound for wce},
\ref{thm:main intro}).
\end{rem}

\subsubsection{\label{subsec:Spectral-theory}Spectral theory}

Recall that the\emph{ spectral zeta function} of the Laplacian $\Delta_{g}$
is defined by
\[
\text{Tr}(\Delta_{g}^{-p}):=\sum_{i=1}^{\infty}\lambda_{i}^{-p}
\]
which is convergent for $p>\dim X/2.$ More precisely, 
\[
\text{Tr}(\Delta_{g}^{-(d/2+\epsilon)})=\frac{1}{\Gamma(d/2)}\frac{\text{Vol}\ensuremath{(X,g)}}{(4\pi)^{d/2}}\frac{1}{\epsilon}+O(1),\,\,\,\epsilon\rightarrow0^{+}
\]
as follows, for example, from the expansion of the heat kernel, i.e.
the Schwartz kernel of $\text{Tr}(e^{-t\Delta_{g}}).$ We will prove
explicit estimates in the case of the two-sphere below. 

The \emph{Fredholm (spectral) determinant }of $\Delta_{g}^{-p}$ is
the function 
\[
D(\lambda,p):=\det(I-\lambda\Delta_{g}^{-p}):=\prod_{i=1}^{\infty}(1-\lambda\lambda_{i}^{-p})
\]
which is convergent for $p>\dim X/2$ and $\lambda\in]0,\lambda_{1}^{p}[.$
Indeed, since the Taylor expansion of $-\log(1-\lambda t)$ equals
$\sum_{m=1}^{\infty}\frac{\lambda^{m}}{m},$ 
\begin{equation}
-\log\det(I-\lambda\Delta_{g}^{-p}):=\sum_{m=1}^{\infty}\frac{\text{Tr}(\Delta^{-mp})}{m}\lambda^{m}\label{eq:taylor expansion}
\end{equation}

\subsubsection{\label{subsec:The-case-of two-sphere sob}The case of the two-sphere}

Consider now the case when $(X,g)$ is the two-sphere, i.e. the unit-sphere
in $\R^{3}$ endowed with the metric $g$ induced by the Euclidean
metric on $\R^{3}.$ Note that under stereographic projection, whereby
$X$ minus the ``north pole'' is identified with $\R^{2},$ we have

\[
\Delta_{g}udV_{g}=-\left(\frac{\partial^{2}u}{\partial^{2}x}+\frac{\partial^{2}u}{\partial^{2}y}\right)dx\wedge dy,\,\,\,dV_{g}=\frac{4dx\wedge dy}{(1+x^{2}+y^{2})^{2}}
\]
Moreover, the set of non-zero eigenvalues of $\Delta_{g}$ are given
by all numbers of the form $l(l+1),$ where $l$ ranges over the positive
integers. The eigenvalue corresponding to a given $l$ has multiplicity
$2l+1.$ 
\begin{rem}
\label{rem:shifted Laplac on sphere}Another convenient norm on $H^{s}(X)/\R$
may be obtained by replacing $\Delta_{g}$ with $\Delta_{g}+1/4$
in formula \ref{eq:def of scalar s} (compare the discussion in Remark
\ref{rem:equiv norms}). The point is that the eigenvalues of $\Delta_{g}+1/4$
are given by $(l+1/2)^{2}.$ This implies that the corresponding spectral
function may be explicitly expressed as $\text{Tr}\left((\Delta_{g}+4^{-1})^{-p}\right):=2^{2p-2}\zeta(2p-1),$
where $\zeta$ is Riemann's zeta function (see \cite[page 453]{vo}). 
\end{rem}

We will use the following slight refinement of \cite[Lemma 26]{b-s-s-w}:
\begin{lem}
\label{lem:compar of w}$1<s'<s$ and $\text{wce }(\boldsymbol{x}_{N};s)\leq1$
then
\[
\text{wce }(\boldsymbol{x}_{N};s')\leq c(s',s)\text{(wce }(\boldsymbol{x}_{N};s)^{s'/s},
\]
 where the constant $c(s',s)$ is given by 
\[
c(s',s)=\sqrt{\frac{\Gamma(s)}{\Gamma(s')}+\frac{s^{s}e^{-s}+1}{\Gamma(s')(s'-1)}}
\]
In particular, $c(2,2+\epsilon)\leq e^{1/2},$ when $\epsilon\leq0.15.$
\end{lem}

\begin{proof}
The main difference with \cite[Lemma 26]{b-s-s-w} is that the constant
$c(s,s')$ in \cite[Lemma 26]{b-s-s-w} depends on a non-explicit
constant $c$ such that for $t\in]0,\epsilon/2[$ the heat kernel
$K_{t}$ satisfies $tK_{t}\leq c\epsilon K_{\epsilon}$ on $X\times X.$
Here we observe that one can, in fact, take $c=1$ and allow $t\in]0,\epsilon[,$
i.e.
\begin{equation}
t\in]0,\epsilon]\implies tK_{t}\leq\epsilon K_{\epsilon}\label{eq:li-yau}
\end{equation}
Indeed, this follows from the Li-Yau parabolic Harnack-inequality
on Riemannian manifolds with non-negative Ricci curvature (apply \cite[Thm 2.3]{l-y}
to $u(x):=K_{t}(x,y)$ for $y$ fixed and $\alpha=1).$ Since the
rest of the proof proceeds essentially as in \cite[Lemma 26]{b-s-s-w},
mutatis mutandis, we will be rather brief. The starting point is the
formula
\begin{equation}
\text{wce }(\boldsymbol{x}_{N};s')^{2}=\frac{1}{\Gamma(s')}\int_{0}^{\infty}t^{s'-1}g(t)dt,\,\,\,g(t)=\frac{1}{N^{2}}\sum_{1\leq i,j\leq N}\mathcal{H}_{t}(x_{i},x_{j})\label{eq:wce in terms of heat}
\end{equation}
where $\mathcal{H}_{t}$ denotes the heat-kernel $K_{t}$ with the
constant term removed, which follows from formula \ref{eq:sobolev is wce}
together with formula \ref{eq:def of scalar s} applied to $-s$ and
the identity $\lambda^{-p}=\int_{0}^{\infty}e^{-t\lambda}t^{p-1}dt/\Gamma(p)$
(note that the formula in \cite{b-s-s-w} corresponding to \ref{eq:wce in terms of heat}
contains a factor $e^{-t}$ due to the different definition of the
Sobolev norms in \cite{b-s-s-w}). Set $\epsilon:=\text{wce }(\boldsymbol{x}_{N};s)^{2/s}(\leq1)$
and split the integral over $t$ above over the three disjoint regions
$]1,\infty[,$ $]\epsilon,1]$ and $]0,\epsilon]$ (in \cite[Lemma 26]{b-s-s-w}
the regions are defined by replacing $\epsilon$ with $\epsilon/2,$
but here we can take $\epsilon$ since we will use the sharper estimate
\ref{eq:li-yau}). First note that
\[
\frac{1}{\Gamma(s')}\int_{1}^{\infty}t^{s'-1}g(t)dt\leq\frac{\Gamma(s)}{\Gamma(s')}\frac{1}{\Gamma(s)}\int_{1}^{\infty}t^{s-1}g(t)dt
\]
and similarly
\[
\frac{1}{\Gamma(s')}\int_{\epsilon}^{1}t^{s'-1}g(t)dt=\epsilon^{s'}\frac{1}{\Gamma(s)}\int_{\epsilon}^{1}(t/\epsilon)^{s'-1}g(t)dt\leq\epsilon^{s'-s}\frac{\Gamma(s)}{\Gamma(s')}\frac{1}{\Gamma(s)}\int_{\epsilon}^{1}t^{s-1}g(t)dt
\]
Since $1\leq\epsilon^{s'-s}$ it follows that 
\[
\frac{1}{\Gamma(s')}\int_{\epsilon}^{\infty}t^{s'-1}g(t)dt\leq\frac{\Gamma(s)}{\Gamma(s')}\epsilon^{s'-s}\frac{1}{\Gamma(s)}\int_{1}^{\infty}t^{s'-1}g(t)dt\leq\frac{\Gamma(s)}{\Gamma(s')}\epsilon^{s'}
\]
Next, using \ref{eq:li-yau}, one gets, precisely as in the proof
of \cite[Lemma 26]{b-s-s-w}, that
\[
\frac{1}{\Gamma(s')}\int_{0}^{\epsilon}t^{s'-1}g(t)dt\leq\epsilon^{s'}\frac{1}{\Gamma(s')(s'-1)}\sup_{t\in[0,1]}g(t)
\]
Finally, one shows essentially as in the proof of \cite[Lemma 26]{b-s-s-w},
that 
\[
0<g(t)\leq t^{s}\sup_{\lambda\in]0,\infty[}\lambda^{s}e^{-\lambda t}=s^{s}e^{-s}
\]
using that the unique maximum is attained at $\lambda=s/t.$ Adding
up the contributions thus concludes the proof of the inequality. Setting
$s'=2$ gives $c(2,s)^{2}=\Gamma(s)+1+s^{s}e^{-s}$ and hence, if
$\epsilon\leq0.15,$ then $c(2,2+\epsilon)\leq\sqrt{1.073+(2.15/e)^{2.15}+1}\leq1.634<e^{1/2}.$ 
\end{proof}
\begin{lem}
\label{lem:bounds on zeta for sphere}On the two-sphere $(X,g)$ the
following inequality holds for any $\epsilon>0$
\[
\text{Tr}(\Delta_{g}^{-(1+\epsilon)})\leq\frac{1}{\epsilon}+2,\,\,\,\text{Tr}(\Delta_{g}^{-2})=1
\]
 and hence 
\[
-\log\det(I-\lambda\Delta_{g}^{-(1+\epsilon)})\leq\lambda\frac{1}{\epsilon}-4\log(1-\frac{\lambda}{2}).
\]
 
\end{lem}

\begin{proof}
Setting $\mathcal{Z}(p):=\text{Tr}(\Delta_{g}^{-p}),$ for $p\geq1,$
we have 
\[
\mathcal{Z}(p)=\sum_{l=1}^{\infty}\frac{2l+1}{l^{p}(l+1)^{p}}=2\sum_{l=1}^{\infty}\frac{1}{l^{p-1}(l+1)^{p}}+\sum_{l=1}^{\infty}\frac{1}{l^{p}(l+1)^{p}}\leq2\sum_{l=1}^{\infty}\frac{1}{l^{p-1}(l+1)^{p}}+1,
\]
 where the second sum was estimated by replacing $p$ with $1$ to
get a a telescoping sum. Next, 
\[
\sum_{l=1}^{\infty}\frac{1}{l^{p-1}(l+1)^{p}}\leq\frac{1}{2}+\sum_{l=2}^{\infty}\frac{1}{l^{p-1}(l+1)^{p}}\leq\frac{1}{2}+\sum_{l=2}^{\infty}\frac{1}{l^{2p-1}}=\frac{1}{2}+\zeta(2p-1)-1
\]
(using the trivial bound $l+1\geq l),$ where $\zeta(s)$ is the Riemann
zeta function. Set $s=1+\delta.$ As is well-known, when $s>1,$ a
resummation argument gives (cf. \cite[formula 3]{g-s})
\[
\zeta(s)\leq\frac{s}{s-1}=1+\frac{1}{\delta},
\]
Hence, setting $p=1+\epsilon,$ gives $\zeta(2p-1)=\zeta(1+2\epsilon)\leq1+1/(2\epsilon).$
All in all, this means that 
\[
\mathcal{Z}(p)\leq2\left(\frac{1}{2}+(1+\frac{1}{2\epsilon})-1\right)+1=2+\frac{1}{\epsilon},
\]
proving the first inequality in the lemma. Next, note that $\mathcal{Z}(2)$
can be computed as a telescoping sum:
\[
\mathcal{Z}(2)=\sum_{l=1}^{\infty}\frac{2l+1}{l^{2}(l+1)^{2}}=\sum_{l=1}^{\infty}\frac{(l+1)^{2}-l^{2}}{l^{2}(l+1)^{2}}=\sum_{l=1}^{\infty}\frac{1}{l^{2}}-\frac{1}{(l+1)^{2}}=1+0.
\]
 Now, Taylor expanding $-\log\det(I-\lambda\Delta_{g}^{-(1+\epsilon)})$
(as in formula \ref{eq:taylor expansion}) gives
\[
-\log\det(I-\lambda\Delta_{g}^{-(1+\epsilon)})=\sum_{m=1}^{\infty}\frac{\lambda^{m}}{m}\text{Tr}(\Delta^{-m(1+\epsilon)})\leq\lambda\left(\frac{1}{\epsilon}+C_{0}\right)+\sum_{m=2}^{\infty}\frac{\lambda^{m}}{m}\text{Tr}(\Delta^{-m(1+\epsilon)}).
\]
 Since the smallest non-negative eigenvalue $\lambda_{1}$ of $\Delta_{g}$
is equal to $2$ we have that, for $m\geq2,$ 
\[
\text{Tr}(\Delta_{g}^{-m(1+\epsilon)})\leq2^{2}2^{-m}\text{Tr}(\Delta_{g}^{-2})\leq2^{2}2^{-m}\cdot1.
\]
As a consequence, 
\[
\sum_{m=2}^{\infty}\frac{\lambda^{m}}{m}\text{Tr}(\Delta_{g}^{-m(1+\epsilon)})\leq2^{2}\sum_{m=2}^{\infty}(\frac{\lambda}{2})^{m}\frac{1}{m}=2^{2}\left(-\log(1-\frac{\lambda}{2})-\frac{\lambda}{2}\right),
\]
 using again that the Taylor expansion of $-\log(1-\lambda t)$ equals
$\sum_{m=1}^{\infty}\frac{\lambda^{m}}{m}.$ All in all, this means
that $-\log\det(I-\lambda\Delta_{g}^{-(1+\epsilon)})$ is bounded
from above by
\[
\lambda\left(\frac{1}{\epsilon}+2\right)-2\lambda-2^{2}\log(1-\frac{\lambda}{2})=\lambda\frac{1}{\epsilon}-4\log(1-\frac{\lambda}{2}),
\]
 
\end{proof}
\begin{rem}
By \cite[Prop 5]{mo} $\text{Tr}(\Delta_{g}^{-(1+\epsilon)})=1/\epsilon+2\gamma-1+o(1)$
as as $\epsilon\rightarrow0^{+},$ where $\gamma=0.577...$ is Euler's
constant. But the point of the previous lemma is to bound the error
term explicitly for $\epsilon$ fixed. 
\end{rem}

\section{Proofs of the main results}

Fix $s>2$ and a positive integer $N.$ Consider the following $H_{0}^{-s}(X)-$valued
random variable on $(X^{N},d\P_{N}):$ 
\[
Y_{N}:=N(\delta_{N}-d\sigma):\,\,\,(X^{N},d\P_{N})\rightarrow H_{0}^{-s}(X),
\]
 where $\delta_{N}$ denotes the empirical measure \ref{eq:emp measure}
(the space $H_{0}^{-s}(X)$ contains the image of $Y_{N}$ for any
$s>\text{1, }$ but the restriction to $s>2$ will turn out to be
important in the following). To keep things as elementary as possible
it will be convenient to consider truncated random variable taking
values in finite dimensional approximations of $H_{0}^{-s}(X)$ (but
a more direct approach could also be used; see Remark \ref{rem:gff}).
To this end fix an orthonormal basis $\nu_{i}$ in the Hilbert space
$\left\langle H_{0}^{-s}(X),\left\langle \cdot,\cdot\right\rangle _{-s}\right\rangle .$
It will be convenient to take $\nu_{i}$ as in Lemma \ref{lem:expansion in terms of eig}
ordered so that $0<\lambda_{1}\leq...\leq\lambda_{M}.$ Let $H_{\leq M}^{-s}(X)$
be the $M-$dimensional subspace of $H_{0}^{-s}(X)$ defined by 
\[
H_{\leq M}^{-s}(X):=\R\nu_{1}\oplus\cdots\oplus\R\nu_{M}\Subset H_{\leq M}^{-s}(X)
\]
Denote by $\pi_{M}$ the orthogonal projection from the Hilbert space
$H_{0}^{-s}(X)$ onto the $M-$dimensional subspace $H_{\leq M}^{-s}(X).$

\subsection*{Step 1: $\pi_{M}(Y_{N})$ is a sub-Gaussian random variable}

We can view $\pi_{M}(Y_{N})$ as a random variable on $(X^{N},d\P_{N})$
taking values in $H_{\leq M}^{-s}(X):$ 
\[
\pi_{M}(Y_{N}):\,\,\,(X^{N},d\P_{N})\rightarrow H_{\leq M}^{-s}(X).
\]
The first step of the proof is to compare $\pi_{M}(Y_{N})$ with a
\emph{Gaussian} random variable $G_{M}$ taking values in $H_{\leq M}^{-s}(X).$
To this end we first endow $H_{\leq M}^{-s}(X)$ with the Hilbert
structure define by the scalar product $\left\langle \cdot,\cdot\right\rangle _{-1}$
and denote by $\gamma_{M}$ the Gaussian measure on $\left\langle H_{\leq M}^{-s}(X),\left\langle \cdot,\cdot\right\rangle _{-1}\right\rangle .$
Concretely, this means that under any linear isometry of $\left\langle H_{\leq M}^{-s}(X),\left\langle \cdot,\cdot\right\rangle _{-1}\right\rangle $
with $\R^{N}$ the measure $\gamma_{M}$ corresponds to the standard
centered Gaussian measure on $\R^{N}.$ Now define $G_{M}$ as a random
element in $H_{\leq M}^{-s}(X).$ In other words, $G_{M}$ is the
random variable defined by the identity map 
\[
G_{M}:=I\,\,\,(H_{\leq M}^{-s}(X),\gamma_{M})\rightarrow H_{\leq M}^{-s}(X)
\]
The following proposition says that the moment generating function
of the random variable $\pi_{M}(Y_{N}),$ viewed as a function on
the linear dual $(H_{\leq M}^{-s}(X))^{*}$ of $H_{\leq M}^{-s}(X)$
is bounded from above by the moment generating function of the scaled
Gaussian random variable $G_{M}/(4\pi)^{1/2}.$
\begin{prop}
\label{prop:sub-gaussian ineq for trunc}The following inequality
holds: 
\begin{equation}
\E(e^{\left\langle \pi_{M}(Y_{N}),\cdot\right\rangle })\leq\E(e^{\left\langle \frac{1}{\sqrt{4\pi}}G_{M},\cdot\right\rangle })\label{eq:sub-gauss wrt gff}
\end{equation}
Equivalently, denoting by $p_{N}^{(M)}$ the law of $\pi_{M}(Y_{N}),$
i.e. the probability measure on $H_{\leq M}^{-s}(X)$ defined by the
push-forward of $d\P_{N}$ under the map $\pi_{M}(Y_{N}),$ we have
\[
L[p_{N}^{(M)}]\leq L[F_{*}\gamma_{M}],\,\,\,\,\,F(v):=\frac{v}{\sqrt{4\pi}}
\]
 where $L[\Gamma]$ denote the Laplace transform of a measure $\Gamma$
on the finite dimensional vector space $V:=H_{\leq M}^{-s}(X),$ i.e.
$L[\Gamma]$ is the function on $V^{*}$ defined by 
\[
L[\Gamma](w)=\int_{V}e^{-w}\Gamma
\]
\end{prop}

\begin{proof}
Applying the Moser-Trudinger type inequality \ref{eq:spherical m-t intro}
for the spherical ensemble proved in \cite{berm3} to $u=w/N+1$ for
$w\in H^{1}(X)$ gives
\[
\E(e^{\left\langle Y_{N},w\right\rangle })\leq e^{\frac{N(N\text{+1)}}{(N+1)(N\text{\textasciiacute+1)}}\frac{1}{2}\frac{1}{4\pi}\left\Vert w\right\Vert _{H^{1}(X)}^{2}}\leq e^{\frac{1}{2}\frac{1}{4\pi}\left\Vert w\right\Vert _{H^{1}(X)}^{2}}
\]
In particular, taking $w\in(H_{\leq M}^{-s}(X))^{*},$ identified
with a subspace of $H^{1}(X)/\R,$ gives $\left\langle Y_{N},w\right\rangle =\left\langle \pi_{M}(Y_{N}),w\right\rangle $
and hence it will be enough to verify that 
\begin{equation}
e^{\frac{1}{2}\left\Vert w\right\Vert _{H^{1}(X)}^{2}}=\E(e^{\left\langle G_{M},w\right\rangle })(=L[\gamma_{N}](-w)\label{eq:Laplace transform of gaussian in particular case}
\end{equation}
 To this end first note that under the identifications above $\left\Vert w\right\Vert _{H^{1}(X)}^{2}$
coincides with the dual norm on the Hilbert space dual of $\left\langle H_{\leq M}^{-s}(X),\left\langle \cdot,\cdot\right\rangle _{-1}\right\rangle .$
But then formula \ref{eq:Laplace transform of gaussian in particular case}
follows from the well-known fact that if $\gamma$ is the Gaussian
measure on a finite dimensional Hilbert space $H,$ then
\begin{equation}
L[\gamma](w)=\exp(\frac{1}{2}\left\Vert w\right\Vert _{H^{*}}^{2}).\label{eq:Laplace transf of gamma for general finite d Hilb}
\end{equation}
 Indeed, fixing an orthonormal basis in $H$ this reduces to the basic
fact that the Laplace transform of the measure $e^{-|x|^{2}/2}dx$
on Euclidean $\R^{N}$ is equal to $e^{|y|^{2}/2},$ which, in turn,
follows from ``completing the square''. 
\end{proof}
\begin{rem}
\label{rem:gff}In the terminology introduced by Kahane, the previous
inequality says that the random variable $\pi_{M}(Y_{N})$ is \emph{sub-Gaussian}
with respect to the Gaussian random variable $\frac{1}{\sqrt{4\pi}}G_{M}.$
In fact, by letting $M\rightarrow\infty$ this implies that $Y_{N}$
is sub-Gaussian with respect to the Laplacian of the Gaussian free
field \cite{she}, viewed as random variables taking values in $H_{0}^{2+\epsilon}(X).$
This point of view will be elaborated on in \cite{berm17}.
\end{rem}

\subsection*{Step 2: Bounding $\E(e^{\left\Vert \pi_{M}(Y_{N})\right\Vert _{-s}^{2}})$}

We start with the following general
\begin{lem}
\label{lem:ineq for exp q in finite dim}Let $H$ be a finite dimensional
Hilbert space and denote by $\gamma$ the corresponding Gaussian measure.
If $\Gamma$ is a measure on $H$ such that the following inequality
of Laplace transforms hold
\[
L[\Gamma]\leq L[\gamma]
\]
as functions on the dual vector space $H^{*},$ then 
\begin{equation}
\int e^{q}\Gamma\leq\int e^{q}\gamma\label{eq:ineq for exp q in lemma}
\end{equation}
 for the squared semi-norm $q(v)$ defined by a given semi-positive
symmetric bilinear form on $H.$ 
\end{lem}

\begin{proof}
First observe that the inequality \ref{eq:ineq for exp q in lemma}
holds for any function $q$ on $H$ which has the following positivity
property: $e^{q}$ is the Laplace transform of a positive measure
$\mu_{q}$ on $H^{*},$ i.e. 
\[
e^{q(v)}=\int_{w\in H^{*}}e^{-\left\langle v,w\right\rangle }d\mu_{q}(w)
\]
Indeed, changing the order of integration (using Fubini) the integral
of $e^{q}$ against $\Gamma$ may be expressed as
\[
\int_{v\in H}\left(\int_{w\in H^{*}}e^{-\left\langle v,w\right\rangle }d\mu_{q}(w)\right)d\Gamma(v)=\int_{w\in H^{*}}\left(\int_{v\in H}e^{-\left\langle v,w\right\rangle }d\Gamma(v)\right)d\mu_{q}(w).
\]
 Hence, by assumption, 
\[
\int_{v\in H}e^{q(v)}d\Gamma(v)\leq\int_{w\in H^{*}}\left(\int_{v\in H}e^{-\left\langle v,w\right\rangle }d\gamma(v)\right)d\mu_{q}(w),
\]
 which is equal to $\int e^{q}\gamma$ (as seen by changing the order
of integration again). All that remains is thus to verify the positivity
property in question when $q$ is a squared semi-norm. Identifying
$H$ with Euclidean $\R^{M}$ and diagonalizing $q$ we may as well
assume that $q=\sum|a_{i}x_{i}|^{2}/2$ for $a_{i}\geq0.$ But then
it follows from formula \ref{eq:Laplace transf of gamma for general finite d Hilb}
and scaling the variables that the measure $d\mu_{q}=\exp\left(-\sum|a_{i}^{-1}y_{i}|^{2}/2\right)dy_{1}...dy_{M}$
has the required property. 
\end{proof}
In the present situation we get the following
\begin{prop}
\label{prop:ineq for truncated norm}For any $\alpha>0$ the following
inequality holds: 
\begin{equation}
\E\left(e^{\alpha\left\Vert \pi_{M}(Y_{N})\right\Vert _{-s}^{2}}\right)\leq\E\left(e^{\frac{\alpha}{4\pi}\left\Vert G_{M}\right\Vert _{-s}^{2}}\right)=\prod_{i\leq M}\left(1-\frac{\alpha}{4\pi}\lambda_{i}^{-(s-1)}\right)^{-1/2}\label{eq:sub-gauss wrt gff-1}
\end{equation}
\end{prop}

\begin{proof}
The first inequality follows directly from combining Prop \ref{prop:sub-gaussian ineq for trunc}
and Lemma \ref{lem:ineq for exp q in finite dim} with $\gamma=F_{*}\gamma_{M}$
$q(v):=\alpha\left\Vert v\right\Vert _{-s}^{2}.$ To prove the last
equality denote by $v_{i}$ an orthonormal base in the Hilbert space
$\left(H_{\leq M}^{-s}(X),\left\langle \cdot,\cdot\right\rangle _{-1}\right).$
Given $v\in H_{\leq M}^{-s}(X)$ we decompose $v=\sum_{i=1}^{M}v_{i}x_{i}$
and note that
\[
\left\Vert v\right\Vert _{-s}^{2}=\sum_{i=1}^{M}x_{i}^{2}\lambda_{i}^{(1-s)}
\]
as follows from writing 
\[
\left\Vert v_{i}\right\Vert _{-s}^{2}:=\left\langle \Delta^{-s}v_{i},v_{i}\right\rangle _{L^{2}}=\left\langle \Delta^{-1}\left(\Delta^{1-s}v_{i}\right),v_{i}\right\rangle _{L^{2}}=\lambda_{i}^{(1-s)}\left\langle \Delta^{-1}v_{i},v_{i}\right\rangle _{L^{2}}=\lambda_{i}^{(1-s)}.
\]
Hence, 
\[
\E\left(e^{\frac{\alpha}{4\pi}\left\Vert G_{M}\right\Vert _{-s}^{2}}\right)=\prod_{i\leq M}\int e^{\frac{\alpha}{4\pi}x_{i}^{2}\lambda_{i}^{(1-s)}}e^{-x_{i}^{2}}dx_{i}/\pi^{1/2}
\]
Finally, changing variables $x_{i}\rightarrow\left(1-\frac{\alpha}{4\pi}\lambda_{i}^{(1-s)}\right)^{1/2}$
in the corresponding Gaussian integrals then concludes the proof of
the proposition.
\end{proof}
Letting $M\rightarrow\infty$ and using the monotone convergence theorem
now concludes the proof of Theorem \ref{thm:moment bound for wce}.

\subsection{Proof of Theorem \ref{thm:moment bound for wce} and Theorem \ref{thm:concentr of measure for spherical}}

Set $s=2+\epsilon$ and $\lambda=\alpha/4\pi.$ Using the Taylor expansion
\ref{eq:taylor expansion}
\begin{equation}
\prod_{i\leq M}\left(1-\lambda\lambda_{i}^{-(s-1)}\right)^{-1/2}\leq-\log\det(I-\lambda\Delta_{g}^{-(1+\epsilon)})\leq\exp\left(\frac{1}{2}\sum_{m=1}^{\infty}\lambda^{m}\frac{\text{Tr}(\Delta^{-m(1+\epsilon)})}{m}\right).\label{eq:bound on D M in pf}
\end{equation}
By Lemma \ref{lem:bounds on zeta for sphere},
\[
\sum_{m=1}^{\infty}\frac{\lambda^{m}\text{Tr}(\Delta^{-m(1+\epsilon)})}{m}\leq\lambda\frac{1}{\epsilon}-4\log(1-\frac{\lambda}{2})=:f(\lambda)
\]
 Hence, for any fixed positive integer $M$ Prop \ref{prop:ineq for truncated norm}
gives 
\[
\E_{N}\left(e^{\alpha\left\Vert \pi_{M}(Y_{N})\right\Vert _{-(2+\epsilon)}^{2}}\right)\leq\exp\left(\frac{1}{2}f(\frac{\alpha}{4\pi})\right)
\]
Letting $M\rightarrow\infty$ and using the monotone convergence theorem
we deduce that 
\[
\E_{N}\left(e^{\alpha\left\Vert (Y_{N})\right\Vert _{-(2+\epsilon)}^{2}}\right)\leq\exp\left(\frac{1}{2}f(\frac{\alpha}{4\pi})\right),
\]
 proving Theorem \ref{thm:moment bound for wce}. Finally, by Chebishev's
inequality, we can write $\P_{N}\left(\left\Vert \delta_{N}-\mu_{\phi}\right\Vert _{H^{-s}}>\delta\right)$
as
\[
\P_{N}\left(\alpha\left\Vert Y_{N}\right\Vert _{H^{-(2+\epsilon)}(X)}^{2}>\alpha\delta^{2}N^{2}\right)\leq e^{-\alpha\delta^{2}N^{2}}\E_{N}\left(e^{\alpha\left\Vert (Y_{N})\right\Vert _{-(2+\epsilon)}^{2}}\right)\leq\exp\left(-\alpha\delta^{2}N^{2}+\frac{1}{2}f(\frac{\alpha}{4\pi})\right)
\]
 Thus, taking any non-zero $\alpha<2\cdot4\pi$ proves the inequality
 in Theorem \ref{thm:concentr of measure for spherical}. 

\subsection{The optimal choice of $\alpha$}

Setting $\lambda:=\alpha/4\pi$ we have 
\begin{equation}
\P_{N}\left(\left\Vert \delta_{N}-\mu_{\phi}\right\Vert _{H(X)^{-(2+\epsilon)}}>\delta\right)\leq\exp\left(\frac{1}{2}\left(-8\pi\delta^{2}N^{2}\lambda+f(\lambda)\right)\right),\,\,\,f(\lambda):=\lambda\frac{1}{\epsilon}-4\log(1-\frac{\lambda}{2}).\label{eq:estimate P in optimal alpha}
\end{equation}
 First observe that $f(\lambda)$ is strictly convex on $[0,2[,$
$f(0)=0$ and $f(2^{-})=\infty.$ Hence, the optimal choice of $\lambda$
satisfies
\begin{equation}
8\pi\delta^{2}N^{2}=f'(\lambda)=\epsilon^{-1}+\frac{2}{1-\lambda/2},\,\,\,\lambda\in]0,2[\label{eq:lambda}
\end{equation}
if such a $\lambda$ exists. Introducing the parameters $R\in]0,\infty[$
and $\eta\in]-1,\infty[$ determined by 
\[
R^{2}:=\delta^{2}N^{2}\epsilon,\,\,\,\eta:=8\pi R^{2}-1
\]
the equation \ref{eq:lambda} for $\lambda$ becomes
\[
\eta=\frac{2\epsilon}{1-\lambda/2}\iff\lambda=2\left(1-\frac{2\epsilon}{\eta}\right)
\]
assuming that $R$ is sufficiently large to ensure that $\lambda\in]0,2[,$
i.e. that 
\[
\eta>2\epsilon.
\]
 As a consequence, for this optimal $\lambda,$ the exponent in the
estimate \ref{eq:estimate P in optimal alpha} becomes

\[
-(8\pi\delta^{2}N^{2}\epsilon)\epsilon^{-1}\lambda+f(\lambda)=-(\eta+1)\epsilon^{-1}\lambda+f(\lambda)=\left(-(\eta+1)+1\right)\epsilon^{-1}\lambda+4\log\left((1-\frac{\lambda}{2})^{-1}\right)=
\]
\[
=2\left(-\eta\right)\left(\epsilon^{-1}-\frac{2}{\eta}\right)+4\log\frac{\eta}{2\epsilon}=2\left(-\eta\right)\epsilon^{-1})+4\left(1+\log\frac{\eta}{2\epsilon}\right)
\]
 In particular, taking $\epsilon:=1/\log N$ gives, for any $\eta>2/\log N,$gives
\begin{equation}
\P_{N}\left(\left\Vert \delta_{N}-\mu_{\phi}\right\Vert _{H(X)^{-(2+1/\log N)}}>R\frac{(\log N)^{1/2}}{N}\right)\leq\label{eq:explicit bound text}
\end{equation}
\[
\leq\exp\left(\frac{1}{2}\left(-8\pi\delta^{2}N^{2}\lambda+f(\lambda\right))\right)=\frac{1}{N^{\eta}}(2\log N)^{2}\exp\left(2\left(1+\log\eta\right)\right).
\]

\subsection{Proof of Theorem \ref{thm:main intro}}

By Lemma \ref{lem:compar of w} it will be enough to prove the inequality
for $s=2.$ Consider the real-valued random variable $W_{N}(s):=\text{wce }(\boldsymbol{x}_{N};s)$
on $(X^{N},d\P_{N}).$ Applying Theorem \ref{thm:concentr of measure for spherical}
to $\delta=R\epsilon^{-1/2}/N$ it will be enough to show the following
claim when $\epsilon:=1/\log N$ under the assumption that $\epsilon^{1/2}\leq R\leq N\epsilon^{1/2}:$
\[
\text{claim:\,}W(2+\epsilon)\leq R\frac{\epsilon^{-1/2}}{N}\implies W(2)\leq CR\frac{\epsilon^{-1/2}}{N},\,\,\,C:=e^{1/2}c(2,2+\epsilon),
\]
 where $c(2,2+\epsilon)$ is defined in Lemma \ref{lem:compar of w}.
To this end recall that, by assumption, $W_{N}(2+\epsilon)\leq1$
and hence by Lemma \ref{lem:compar of w} we have, since $\epsilon\leq1$
\[
W(s)\leq cW(2+\epsilon)^{\frac{s}{2+\epsilon}},\,\,\,c=c(2,2+\epsilon)
\]
In particular, 
\[
W(2)\leq cW(2+\epsilon)^{\frac{2}{2+\epsilon}}\leq cW(2+\epsilon)^{(1-\epsilon/2)}
\]
using that $1/(1+t)\geq1-t$ if $t\geq0$ and that $W_{N}(2+\epsilon)\leq1.$
Hence, $W_{N}(2+\epsilon)\leq R\frac{\epsilon^{-1/2}}{N}$ implies
that 
\[
W(2)\leq c(R\frac{\epsilon^{-1/2}}{N})^{(1-\epsilon/2)}=cR\frac{\epsilon^{-1/2}}{N}(R\frac{\epsilon^{-1/2}}{N})^{-\epsilon/2}.
\]
 But, by assumption, $R\epsilon^{-1/2}\geq1$ and hence 
\[
(R\frac{\epsilon^{-1/2}}{N})^{-\epsilon/2}=(R\epsilon^{-1/2})^{-\epsilon/2}N^{\epsilon/2}\leq N^{\epsilon/2}:=\left(N^{1/(\log N)}\right)^{1/2}=e^{1/2},
\]
since $\epsilon:=1/\log N.$ 
\begin{rem}
If, in particular, $1/\log N\leq0.15$ (for, example, $N=1000)$ then
we get, by Lemma \ref{lem:compar of w}, $W(2)\leq e^{1/2}e^{1/2}R(\frac{\log N}{N})^{1/2}\leq eR(\frac{\log N}{N})^{1/2}.$
\end{rem}

\subsection{Explicit formulation of Theorem \ref{thm:main intro}}

Combining the previous remark with the explicit bound in formula \ref{eq:explicit bound text}
gives the following explicit formulation of the first inequality in
Theorem \ref{thm:main intro} 
\[
\P_{N}\left(\text{wce }(\boldsymbol{x}_{N};2)\leq e\sqrt{\frac{1+\eta}{8\pi}}\frac{(\log N)^{1/2}}{N}\right)\geq1-\frac{(2e\eta\log N)^{2}}{N^{\eta}},
\]
 under the assumption that $\eta>2/\log N$ and $N\geq1000.$

\end{document}